\documentclass[12pt]{amsart}
\usepackage[colorlinks=true,pagebackref,hyperindex,citecolor=green,linkcolor=red]{hyperref}
\usepackage{fullpage}
\usepackage{setspace}
\usepackage{amsmath}
\usepackage{amsfonts}
\usepackage{amssymb}
\usepackage{bm}
\usepackage{mathrsfs}
\usepackage{color}
\usepackage[all]{xy}
\usepackage{graphicx}
\usepackage{booktabs}

\usepackage{mdwlist}

\usepackage{verbatim}

 \theoremstyle{definition}

 \newtheorem{Theorem}{Theorem}[section]
 \newtheorem{Corollary}[Theorem]{Corollary}
 \newtheorem{Lemma}[Theorem]{Lemma}
 \newtheorem{Proposition}[Theorem]{Proposition}
 \newtheorem{Definition}[Theorem]{Definition}
  \newtheorem{Remark}[Theorem]{Remark}

 \newtheorem{Example}[Theorem]{Example}
 \newtheorem{Notation}[Theorem]{Notation}

\newtheorem*{DiagonalFPT1}{Theorem \ref{DiagonalFPT: T}:  Part I}
\newtheorem*{DiagonalFPT2}{Theorem \ref{DiagonalFPT: T}:  Part II}

\newtheorem*{DTI1}{Theorem \ref{DiagonalTestIdeal: T}:  Part I}
\newtheorem*{DTI2}{Theorem \ref{DiagonalTestIdeal: T}:  Part II}
\newtheorem*{DTI3}{Theorem \ref{DiagonalTestIdeal: T}:  Part III}

\newtheorem*{JN1}{Theorem \ref{JumpingNumbers: T}:  Part I}
\newtheorem*{JN2}{Theorem \ref{JumpingNumbers: T}:  Part II}

\newtheorem*{FermatFPTCorollary}{Corollary \ref{FermatFPT: C}}

\newcommand{\up}[1]{\left\lceil #1 \right\rceil}

\newcommand{\mfrak}[1]{\mathfrak{#1}}

\renewcommand{\k}{\boldsymbol{k}}

\newcommand{\ve}{\boldsymbol{\text{v}}}
\newcommand{\uu}{\boldsymbol{u}}

\newcommand{\x}{\boldsymbol{x}}

\newcommand{\m}{\mfrak{m}}

\renewcommand{\bold}[1]{\mathchoice{\hbox{\boldmath $\displaystyle #1$}}
        {\hbox{\boldmath $\textstyle #1$}}
        {\hbox{\boldmath $\scriptstyle #1$}}
        {\hbox{\boldmath $\scriptscriptstyle #1$}}}

\newcommand{\fpt}[1]{\bold{\operatorname{fpt}}_{\mathfrak{m}}(#1)}
\newcommand{\tr}[2]{\left \langle {#1} \right \rangle_{#2}}

\newcommand{\test}[3]{\bold{\tau} \left({#3} \bullet {#2} \right)}

\newcommand{\pideal}[3]{\left( {#1}^{#2} \right)^{\left [\frac{1}{p^{#3}} \right]}}

\newcommand{\trd}[2]{\tr{ \frac{1}{d_{#1}}}{#2}}

\newcommand{\bracket}[2]{ {#1}^{ \left [ p^{#2} \right]}}

\newcommand{\error}{\varepsilon}

\newcommand{\base}{\operatorname{base}}

\newcommand{\0}{\bold{0}}

\renewcommand{\l}{\ell}

\newcommand{\set}[1]{ \left\{ \, #1 \, \right\}}

\renewcommand{\(}{\left(}
\renewcommand{\)}{\right)}

\newcommand{\emw}{\omega}
\newcommand{\e}{e}

\renewcommand{\th}{\text{th}}

\newcommand{\Dee}{\bold{\operatorname{D}}}
\newcommand{\DDelta}{\bold{\Delta}}
\newcommand{\ddelta}{\bold{\delta}}

\newcommand{\vleq}{\preccurlyeq}
\newcommand{\vl}{\prec}

\newcommand{\vgeq}{\succcurlyeq}

\newcommand{\Coeff}[3]{\Gamma_{#1}^{#3} \( #2 \)}

\newcommand{\dotp}{\bullet}
\newcommand{\digit}[2]{ #1^{\(#2\)} }

\newcommand{\Basis}[1]{\mathscr{B}_{#1}}
\newcommand{\vone}{\bold{1}}
\renewcommand{\vee}{\bold{\operatorname{v}}}

\newcommand{\MPT}{Main Theorem}
\renewcommand{\ell}{l}

\newcommand{\gfpt}[1]{\bold{\operatorname{fpt}}\(#1\)}

\newcommand{\llam}{\bold{\lambda}}

\newcommand{\cc}{\bold{\operatorname{c}}}

\begin{document}

\title{$F$-invariants of diagonal hypersurfaces}
\author{ Daniel Jes\'us Hern\'andez }
\thanks{The author was partially supported by the National Science Foundation RTG grant number 0502170 at the University of Michigan.}

\begin{abstract}  In this note, we derive a formula for the $F$-pure threshold of diagonal hypersurfaces over a perfect field  of prime characteristic.  We also calculate the associated test ideal at the $F$-pure threshold, and give formulas for higher jumping numbers of Fermat hypersurfaces.
\end{abstract}

\maketitle

\section*{Introduction}

Let $R$ be a polynomial ring over a perfect field $\mathbb{L}$ of characteristic $p>0$, and consider a polynomial $f \in R$.  Using the \emph{Frobenius} morphism $R \to R$ given by $r \mapsto r^p$, one may define a family of ideals $\set{ \test{R}{f}{\lambda} \subseteq R : \lambda > 0}$ called the \emph{test ideals} of $f$.  Test ideals (defined in the context of \emph{tight closure}) were originally introduced in \cite{HH1990}, and generalized to pairs in \cite{HY2003}.  Test ideals vary with respect to $\lambda$ in the following way:  they shrink as $\lambda$ increases, and are also stable to the right.  We say that a parameter $\lambda$ is an \emph{$F$-jumping number} of $f$ if $\test{R}{f}{\lambda} \neq \test{R}{f}{\( \lambda - \error\)}$ for every $0 < \error < \lambda$.  We call the smallest $F$-jumping number the \emph{$F$-pure threshold} of $f$ and denote it by $\gfpt{f}$.  In this article, we consider these invariants when $f$ is \emph{diagonal} or \emph{Fermat}.  Recall that $f$ is called {diagonal} if it is an $\mathbb{L}^{\ast}$-linear combination of $x_1^{d_1}, \cdots, x_n^{d_n}$ and \emph{Fermat} if it is an $\mathbb{L}^{\ast}$-linear combination of $x_1^d, \cdots, x_d^d$.

In Theorem \ref{DiagonalFPT: T}, we give a formula for the $F$-pure threshold of a diagonal hypersurface as a function of the characteristic.  In Theorem \ref{DiagonalTestIdeal: T}, we give a formula for the first non-trivial test ideal of a diagonal hypersurface.  Note that (classical) test ideals of diagonal hypersurfaces were computed by McDermott in \cite{McDermott1} and \cite{McDermott2}.  In Theorem \ref{JumpingNumbers: T}, we give conditions for the existence of, and formulas for, higher jumping numbers of Fermat hypersurfaces.  For a detailed discussion of our main results, and for examples, see Section \ref{Discussion: S}.

\subsection{Acknowledgements} This work is part of the author's Ph.D. thesis at the University of Michigan.  I would like to thank Karen Smith for suggesting this problem, as well as Emily Witt,  whose observation led to the statement and proof of Theorem \ref{DiagonalFPT: T}.

\section{Test Ideals and $F$-pure thresholds}
\label{TI: S}

Let $\mathbb{L}$ be a a perfect of characteristic $p>0$, and let $R=\mathbb{L}[x_1, \cdots, x_n]$.  We will use $\m$ to denote the ideal $\( x_1, \cdots, x_n\)$.  As $\mathbb{L}$ is perfect, we have that $R^{p^e} := \mathbb{L}[x_1^{p^e}, \cdots, x_n^{p^e}]$ is the subring of $\(p^e\)^{\th}$ powers of $R$.  For every ideal $I \subseteq R$, let $\bracket{I}{e}$ denote the ideal generated by the set $\set{ g^{p^e} : g \in I}$.  We call $\bracket{I}{e}$ the $e^{\th}$ \emph{Frobenius power} of $I$.  

\begin{Definition}
\label{Basis: D}
We will use $\Basis{e}$ to denote the set of monomials $\set{ \mu : \mu \notin \bracket{\m}{e}}$. The reader may verify that 
$\Basis{e}$ is a free basis for $R$ as an $R^{p^e}$-module.  If $f \in R$ is a non-zero polynomial and $\mu \in \Basis{e}$, we use $\Coeff{\mu}{f}{e}$ to denote the element of $R$ such that $f = \sum_{\mu \in \Basis{e}} \Coeff{\mu}{f}{e}^{p^e} \mu$.
\end{Definition}

\begin{Remark}
\label{Flatness: R}
As $R$ is finitely generated and free over $R^{p^e}$, it follows that $f^{p^e} \in \bracket{I}{e}$ if and only if $f \in I$.
\end{Remark}



\begin{Definition} 
\label{Pideal: D}
Let $f \in R$ be a non-zero polynomial.  We use $\pideal{f}{}{e}$ to denote the ideal generated by the set $\set{\Coeff{\mu}{f}{e} : \mu \in \Basis{e} }$.
\end{Definition}



Lemma \ref{LittleL: L} follows from \cite[Proposition 2.5]{BMS2008}, though we include a proof for the sake of completeness.

\begin{Lemma} 
\label{LittleL: L}
Let $f \in R$.  If $I \subseteq R$ is an ideal, then $\pideal{f}{}{e} \subseteq I$ if and only if $f \in \bracket{I}{e}$. 
\end{Lemma}

\begin{proof} If $\pideal{f}{}{e} \subseteq I$, then $f \in \bracket{\( \pideal{f}{}{e} \)}{e} \subseteq \bracket{I}{e}$.  Instead, suppose  $f \in \bracket{I}{e} = (a_1^{p^e}, \cdots, a_s^{p^e})$.  Then, $f = \sum_{i=1}^s g_i \cdot a_i^{p^e} =  \sum_{\mu \in \Basis{e}} \( \sum_{i=1}^s a_i \Coeff{\mu}{g_i}{e} \)^{p^e} \mu$.  Thus, $\Coeff{\mu}{f}{e} = \sum_{i=1}^s a_i \Coeff{\mu}{g_i}{e}$, and we conclude that $\pideal{f}{}{e} \subseteq I$. 
\end{proof}

\begin{Remark}
Lemma \ref{LittleL: L} shows that $\pideal{f}{}{e}$ is the unique minimal ideal $I$ such that $f \in \bracket{I}{e}$.  This shows that $\pideal{f}{}{e}$ does not depend on the specific choice of basis $\Basis{e}$ for $R$ over $R^{p^e}$.
\end{Remark}

\begin{Definition}
\label{TestIdealDefinition} For every $\lambda \geq 0$, the set $\set{ \pideal{f}{\up{p^e \lambda}}{e} : e \geq 1}$ defines an increasing sequence of ideals \cite[Lemma 2.8]{BMS2008}.  We call the stabilizing ideal the \emph{test ideal} of $f$ (with respect to the parameter $\lambda$), and denote it by $\test{R}{f}{\lambda}$.  In other words, \[ \test{R}{f}{\lambda} =  \bigcup_{e \geq 1}  \pideal{f}{\up{ p^e \lambda }}{e} = \pideal{f}{\up{p^e \lambda}}{e} \text{ for all } e \gg 0. \]
\end{Definition}

The following lemma, whose proof we omit, allows us to identify when the test ideal stabilizes in an important special case.

\begin{Lemma}\cite[Lemma 2.1]{BMS2009}
\label{Stabilization: L}
If $\lambda \in \frac{1}{p^e} \cdot \mathbb{N}$, then $\test{R}{f}{\lambda} = \pideal{f}{p^e \lambda}{e}$.
\end{Lemma}

Test ideals form a decreasing sequence of ideals, and are stable to the right \cite[Proposition 2.11, Corollary 2.16]{BMS2008} .  That is, $\test{R}{f}{\lambda} \subseteq \test{R}{f}{\lambda_{\circ}}$ if $\lambda \geq \lambda_{\circ}$.  Additionally, for every $\lambda \geq 0$ there exists $\error > 0$ such that $\test{R}{f}{\lambda} = \test{R}{f}{\(\lambda+\delta\)}$ whenever $0 \leq \delta < \error$.  This behavior motivates the following definition.

\begin{Definition}  We say that $\lambda>0$ is an \emph{$F$-jumping number} of $f$ if \[ \test{R}{f}{\lambda} \neq \test{R}{f}{\( \lambda - \error \)} \text{ for all } 0 < \error <  \lambda.\]  
\noindent By convention, we consider  $0$ an $F$-jumping number of $f$. 
\end{Definition}

\begin{Proposition} \cite[Proposition 2.25]{BMS2008}
\label{JumpingNumbersinInterval: P}
A number $\gamma > 1$ is an $F$-jumping number of $f$ if and only if $\gamma - 1$ is an $F$-jumping number of $f$.  \end{Proposition}


Let $f$ be a non-zero, non-unit polynomial in $R$, and choose $e \gg 0$ so that $p^e > \deg f$.  It follows that for every proper ideal $I \subsetneq R$,  we have that $f \notin \bracket{I}{e}$.  This, combined with Lemma \ref{LittleL: L} and Lemma \ref{Stabilization: L}, shows that $\pideal{f}{}{e} = \test{R}{f}{\frac{1}{p^e}}$ is not contained in any proper ideal of $R$, and thus must equal $R$.  We see that $\test{R}{f}{\lambda} = R$ for $0 < \lambda \ll 1$, and so the smallest non-zero $F$-jumping number of $f$ is the minimal parameter $\lambda$ such that $\test{R}{f}{\lambda} \neq R$.  This jumping number is of particular interest, and is called the $F$-pure threshold of $f$. 

\begin{Definition}
\label{FPTDefinition} We call $\gfpt{f}: = \sup \set{ \lambda \in \mathbb{R}_{\geq 0} : \test{R}{f}{\lambda} = R }$ the \emph{$F$-pure threshold} of $f$, and we call $\fpt{f} :=   \sup \set{ \lambda \in \mathbb{R}_{\geq 0} : \test{R}{f}{\lambda}_{\m} = R_{\m} } $ the $F$-pure threshold of $f$ at $\m$.
\end{Definition}

In our computations, we will use the following well known description of $\fpt{f}$.

\begin{Lemma}
\label{TypicalFPT: L}
$\fpt{f} = \max \set{ \lambda >  0 : \exists \ e_{\lambda} \geq 1\text{ with } f^{\up{p^e \lambda}} \notin \bracket{\m}{e} \text{ for all } e \geq e_{\lambda}}$.
\end{Lemma}

\begin{proof}
Comparing with Definition \ref{FPTDefinition}, we see  it suffices to show $\test{R}{f}{\lambda}_{\m} = R_{\m}$ if and only if there exists $e_{\lambda} \geq 1$ with $f^{\up{p^e}} \notin \bracket{\m}{e}$ for all $e \geq e_{\lambda}$.  By definition, there exists $e_{\lambda}$ such that $\test{R}{f}{\lambda} = \pideal{f}{\up{p^e \lambda}}{e}$ for all $e \geq e_{\lambda}$.  For such an $e$, $\test{R}{f}{\lambda}_{\m} = R_{\m}$ if and only if $\pideal{f}{\up{p^e \lambda}}{e}_{\m} = R_{\m}$.  However, this occurs if and only if $\pideal{f}{\up{p^e \lambda}}{e} \not \subseteq \m$, which Lemma \ref{LittleL: L} shows happens if and only if $f^{\up{p^e \lambda}} \notin \bracket{\m}{e}$.
\end{proof}

\section{Some remarks on base $p$ expansions}
\label{Expansions: SS}

\begin{Definition}
 Let $\alpha \in (0,1]$, and let $p$ be a prime number.  Let $\digit{\alpha}{d}$ be the unique integer in $[0,p-1]$ such that $\alpha = \sum_{d \geq 1}\frac{\digit{\alpha}{d}}{p^d}$ and such that $\digit{\alpha}{d} \neq 0$ is not eventually zero as a function of $d$.  We call $\digit{\alpha}{d}$ the $d^{\th}$ \emph{digit} of the non-terminating base $p$ expansion of $\alpha$.  We adopt the convention that $\digit{\alpha}{0} = \digit{0}{d} = 0$.  
\end{Definition}

\begin{Example}
If $\alpha = \frac{1}{p} = \frac{0}{p} + \sum_{e \geq 2}$, we see that $\digit{\alpha}{1} = 0$ and $\digit{\alpha}{e} = p-1$ for all $e \geq 1$.
\end{Example}


\begin{Definition}
\label{TruncationDefinition}
If $\lambda \neq 0$, we call $\tr{\lambda}{e} : = \sum_{d=1}^e \frac{\digit{\lambda}{d}}{p^d}$ the $e^{\th}$ \emph{truncation} of $\lambda$ (in base $p$).  
\end{Definition}

\begin{Lemma}
\label{Truncation: L} If $\lambda \in [0,1]$, then $\up{p^e \lambda} = p^e \tr{\lambda}{e} + 1$. Furthermore, if $\alpha \in [0,1] \cap \frac{1}{p^e} \cdot \mathbb{N}$ and $\lambda > \alpha$, then $\tr{\lambda}{e} \geq \alpha$.  \end{Lemma}

\begin{proof}  As $p^e \lambda = p^e \tr{\lambda}{e} + p^e \cdot \sum_{d > e} \frac{\digit{\lambda}{d}}{p^d}$, the first claim follows from the observation that $0 <  \sum_{d >e} \frac{\digit{\lambda}{d}}{p^d} \leq \frac{1}{p^e}$.  We also see that $\frac{1}{p^e} + \tr{\lambda}{e} \geq \lambda > \alpha$, so \begin{equation} \label{rounding: e} 1 + p^e \tr{\lambda}{e} > p^e \alpha.\end{equation}  By hypothesis, both sides of \eqref{rounding: e} are integers, and we conclude that $p^e \tr{\lambda}{e} \geq p^e \alpha$.
\end{proof}




\begin{Definition}
\label{CarryingDefinition}
Let $(\lambda_1, \cdots, \lambda_n) \in [0,1]^n$, and let $p$ be a prime number.
We say the $e^{\th}$ digits of  $\lambda_1, \cdots, \lambda_n$ \emph{add without carrying} (in base $p$) if $\digit{\lambda_1}{e} + \cdots + \digit{\lambda_n}{e} \leq p-1$, and we say that $\lambda_1, \cdots, \lambda_n$ add without carrying if all of their digits add without carrying.  We say natural numbers $k_1, \cdots, k_n$ add without carrying (in base $p$) if the obvious condition holds.
\end{Definition}


\begin{Remark}
\label{AddingDigits: R}
If $\lambda_1, \cdots, \lambda_n$ add without carrying (in base $p$) and $\lambda:= \sum_{i=1}^n \lambda_i \leq 1$, then $\digit{\lambda}{e} = \digit{\lambda_1}{e} + \cdots + \digit{\lambda_n}{e}$ for all $e \geq 1$.
\end{Remark}

The notion of adding without carrying is relevant in light of the following classical result.

\begin{Lemma}\cite{Dickson, Lucas} 
\label{Lucas: L}  
 Let $\k = (k_1, \cdots k_n) \in \mathbb{N}^n$ and set $N= | \k | = \sum k_i$.  Then $\binom{N}{\k} := \frac{ N! }{k_1 ! \cdots k_n!} \not \equiv 0 \mod p  $ if and only if  $k_1, \cdots, k_n$ add without carrying (in base $p$).
\end{Lemma}

\section{Discussion of the main results}
\label{Discussion: S}

\subsection{$F$-pure theshholds of diagonal hypersurfaces}
In our first result, we give a formula the $F$-pure threshold of a diagonal hypersurface.

\begin{Theorem}
\label{DiagonalFPT: T}
Let $L$ be the supremum over all $N$ such that the $e^{\th}$ digits of $\frac{1}{d_1}, \cdots, \frac{1}{d_n}$ add without carrying for all $0 \leq e \leq N$.  If $f$ is a $\mathbb{L}^{\ast}$-linear combination of $x_1^{d_1}, \cdots,  x_n^{d_n}$, then 
\[ \fpt{f} = \begin{cases} \ \frac{1}{d_1} + \cdots + \frac{1}{d_n} & \text{if } L = \infty  \vspace{.1in} \\   \trd{1}{L}+ \cdots + \trd{n}{L}+ \frac{1}{p^L} & \text{if } L < \infty \end{cases} \]
\end{Theorem}

Formulas for the $F$-pure threshold of $x^2 + y^3$ and $x^2 +y^7$ are given in \cite
[Example $4.3$ and $4.4$]{MTW2005}.  At first glance, these formulas appear to be quite different from those in Theorem \ref{DiagonalFPT: T} above.   Below, we show how Theorem \ref{DiagonalFPT: T} may be used to recover these formulas.

\begin{Example}
\label{FPT: E}
We adopt decimal notation for base $p$ expansions. For example, if $a, b$ are integers with $0 \leq a,b \leq p-1$, then  $. \overline{ a \ b } \( \base p \)$ will denote the unique number $\lambda$ with the property that $\digit{\lambda}{e} = a$ for $e$ odd and $\digit{\lambda}{e} = b$ for $e$ even.  
Let $f$ be a $\mathbb{L}^{\ast}$-linear combination of $x^2$ and $y^3$.  If $p=3$, then
\[\frac{1}{2} = . \overline{1} \ (\base 3) \text{ and } \frac{1}{3} = . 1 = . 0 \ \overline{2} \ (\base 3). \]

We see that carrying is required to add the second digits of $\frac{1}{2}$ and $\frac{1}{3}$ (but not the first), and Theorem \ref{DiagonalFPT: T} implies $\fpt{f} = \tr{\frac{1}{2}}{1} + \tr{\frac{1}{3}}{1} + \frac{1}{3} = 0 + \frac{1}{3} + \frac{1}{3} = \frac{2}{3}$.  Similarly, one can show that $\fpt{f} = \frac{1}{2}$ if $p = 2$. If $p = 6 \emw + 1$ for some $\emw \geq 1$, then
\[ \frac{1}{2} = . \overline{3 \emw } \ (\base p) \text{ and } \frac{1}{3} = . \overline{2 \emw} \ (\base p). \]

We notice that $\frac{1}{2}$ and $\frac{1}{3}$ add without carrying (in base $p$), and Theorem \ref{DiagonalFPT: T} implies $\fpt{f} = \frac{1}{2}+ \frac{1}{3} = \frac{5}{6}$.  Finally, if $p = 6 \emw + 5$ for some $\emw \geq 0$, then 
\[ \frac{1}{2} = . \overline{3 \emw +2} \ (\base p) \text{ and } \frac{1}{3} = . \overline{ 2 \emw+1 \hspace{.1in}  \ 4 \emw+3} \ (\base p). \]
Once more, we see that carrying is needed to add the second digits of $\frac{1}{2}$ and $\frac{1}{3}$, (but not the first), and Theorem \ref{DiagonalFPT: T} implies \begin{equation} \label{fptex: e} \fpt{f} = \tr{\frac{1}{2}}{1} + \tr{\frac{1}{3}}{1} + \frac{1}{p} =  \frac{3 \emw +2}{p} + \frac{2 \emw + 1}{p} + \frac{1}{p} = \frac{5 \emw + 4}{p}.\end{equation}  The reader may verify that $\frac{5 \emw + 4}{p} + \frac{1}{6p} = \frac{5}{6}$, so we may rewrite \eqref{fptex: e} as $\fpt{f} = \frac{5}{6} - \frac{1}{6p}$.  Thus, we recover the following formula from \cite[Example 4.3]{MTW2005}:  \[ \fpt{x^2+y^3} = \begin{cases} 1/2 & \text{if } p=2 \\ 2/3 & \text{if } p=3 \\ 5/6 & \text{if } p \equiv 1 \bmod 6 \\ \frac{5}{6}-\frac{1}{6p} & \text{if } p \equiv 5 \bmod 6 \end{cases}.\] 
\end{Example}


\subsection{A computation of the first non-trivial test ideal}

Our second theorem computes the value of the test ideal at the $F$-pure threshold.

\begin{Theorem}
\label{DiagonalTestIdeal: T}  If $f$ is a $\mathbb{L}^{\ast}$-linear combination of $x_1^{d_1}, \cdots,  x_n^{d_n}$, then 
\[ \test{R}{f}{\fpt{f}} = \begin{cases} (f) & \text{if } \ \fpt{f}=1 \\
                                                             \m & \text{if } \ \fpt{f}= \frac{1}{d_1} + \cdots + \frac{1}{d_n} \\
                                                             \m & \text{if } \ \fpt{f} < \min \set{1, \sum_{i=1}^n \frac{1}{d_i}} \text{ and }  p > \max \set{d_1, \cdots, d_n}.
                                                             \end{cases} \]
\end{Theorem}

\begin{Remark}  Note that $\test{R}{f}{\fpt{f}}$ need not equal $\m$ if  $\fpt{f} < \min \set{1, \sum_{i=1}^n \frac{1}{d_i}}$ and $p$ is less than or equal to some exponent \cite[Proposition 4.2]{MY2009}.
\end{Remark}

\subsection{On (higher) $F$-jumping numbers of Fermat hypersurfaces}
Our final result computes higher jumping numbers of the degree $d$ Fermat hypersurface.  By Proposition \ref{JumpingNumbersinInterval: P}, it suffices to only consider those jumping numbers contained in $(0,1]$.  Theorem \ref{DiagonalFPT: T} takes the following simple form when $f$ is the degree $d$ Fermat hypersurface.

\begin{Corollary}
\label{FermatFPT: C}
 If $f$ is a $\mathbb{L}^{\ast}$-linear combination of $x_1^d, \cdots, x_d^{d}$, then
\[ \fpt{f} = \begin{cases} \frac{1}{p^{\ell}} & \text{if } p^{\ell} \leq d < p^{\ell + 1} \text{ for some } \ell \geq 1 \\ 1 - \frac{a-1}{p} & \text{if } 0 < d < p \text{ and } p \equiv a \bmod d \text{ with } 1 \leq a < d \end{cases} \]
\end{Corollary}

\begin{Theorem}
\label{JumpingNumbers: T}
Suppose that $p>d$ and write $p = d \cdot \emw  + a$ for some $\emw \geq 1$ and $1 \leq a < d$.  If $f$ is a $\mathbb{L}^{\ast}$-linear combination of $x_1^{d_1}, \cdots,  x_n^{d_n}$ and $a = 1$, Corollary \ref{FermatFPT: C} implies that $\fpt{f} =1$.  We now assume $a \geq 2$.
\begin{enumerate}
\item If $p< a(d-1)$, then $\fpt{f} <  \frac{(d+1) \cdot \emw + \up{2a/d}}{p} \leq 1$ are $F$-jumping numbers in $(0,1]$.
\item If $p > a(d-1)$, then $\fpt{f} < 1$ are the only $F$-jumping numbers in $(0,1]$.
\end{enumerate}
\end{Theorem}

\begin{Remark} As $a$ is strictly less than $d$, Theorem \ref{JumpingNumbers: T} implies that $\fpt{f}$ and $1$ are the only jumping numbers of $f$ in $(0,1]$ if $p > (d-1)^2$.
\end{Remark}

\begin{Example}  Suppose that $d=4$, and $p=7$.  Then $\emw=1$, $a=3$, and $p < a(d-1)$.  We see that $(d+1) \cdot \emw + \up{2a/d} = 5 + \up{6/4} = 7 = p$.  In this case, Theorem \ref{JumpingNumbers: T} provides no new information.
\end{Example}

\begin{Example}  Instead, let $d=6$ and $p=11$, so that $\emw =1, a=5$, and $p < a (d-1)$.  We see that $(d+1) \cdot \emw + \up{2a/d} = 7 + \up{10/5}= 9$.  Corollary \ref{FermatFPT: C} and Theorem \ref{JumpingNumbers: T} then imply $\fpt{f} = 1-\frac{a-1}{p} = \frac{7}{11}, \frac{(d+1) \cdot \emw + \up{2a/d}}{p} = \frac{9}{11}, \text{ and } 1$ are $F$-jumping numbers of $f$ contained in $(0,1]$.  The reader may verify that these are \emph{all} of the $F$-jumping numbers of $f$ in $(0,1]$
\end{Example}

\section{$F$-pure thresholds of diagonal hypersurfaces}
\label{DiagonalFPT: S}

\begin{Notation}
\label{DeltaNotation: N}
Set $\delta_i = \frac{1}{d_i}$,  $\ddelta := (\delta_1, \cdots, \delta_n)$, and $\tr{\ddelta}{e}: = (\tr{\delta_1}{e}, \cdots, \tr{\delta_n}{e})$.  As in the statement of Theorem \ref{DiagonalFPT: T}, $L = \sup \set{ N : \digit{\delta_1}{e} + \cdots + \digit{\delta_n}{e} \leq p-1 \text{ for all } 0 \leq e \leq N}$.

$\Dee$ will denote the diagonal matrix whose $i^{\th}$ diagonal entry is $d_i$, and we set $\DDelta : =\Dee^{-1}$.   Note that $\DDelta$ is also diagonal, with the $i^{\th}$ diagonal entry being $\delta_i$.    Throughout this chapter, we assume that $f$ is a $\mathbb{L}^{\ast}$-linear combination of $x_1^{d_1}, \cdots,  x_n^{d_n}$, and write $f = u_1 x^{d_1} + \cdots + u_n x^{d_n}$.   Using multi-index notation, \begin{equation} \label{expansionf: e} f^N = \sum_{| \k | = N} \binom{N}{\k} \uu^{\k} \x^{ \Dee \k}. \end{equation}

If $\bold{\lambda} \in \mathbb{R}^n$, we use $| \bold{\lambda} |$ to denote the coordinate sum $\lambda_1 + \cdots + \lambda_n$.  When considering elements of $\mathbb{R}^n$,  $\vleq$  (and $\vl$) will denote component-wise (strict) equality.   Finally,  $\set{\vee_1, \cdots, \vee_n}$ denotes the standard basis of $\mathbb{R}^n$, and $\vone_n: = (1, \cdots, 1)$.
\end{Notation}

Though the first part of Theorem \ref{DiagonalFPT: T} follows directly from a more general statement from \cite{Polynomials}, we have included a proof below in this simple case.  

\begin{DiagonalFPT1}
If $L = \infty$, then $\fpt{f} = |\ddelta|$.
\end{DiagonalFPT1}

\begin{proof}
Suppose that $f^{\up{p^e \lambda}} \notin \bracket{\m}{e}$.   By \eqref{expansionf: e}, there exists $\k \in \mathbb{N}$ with $| \k | = \up{p^e \lambda}$ and $\Dee \k \vl (p^e-1) \cdot \vone_n$,  so that $\k \vl (p^e-1) \cdot \ddelta$.  Thus, $p^e \lambda \leq \up{p^e \lambda} = | \k | < (p^e-1) | \ddelta |$, and so $\lambda < | \ddelta|$.  It follows from Lemma \ref{TypicalFPT: L} that $\fpt{f} \leq | \ddelta|$.

As $L = \infty$, the entries of $\ddelta$ add without carrying (in base $p$), and it follows that the entries of $p^e \tr{\ddelta}{e}$ add without carrying for all $e \geq 1$.  By Lemma \ref{Lucas: L}, $\binom{p^e | \tr{\ddelta}{e} |}{p^e \tr{\ddelta}{e}} \neq 0 \bmod p$, and as $\Dee \tr{\ddelta}{e} \vl \Dee \ddelta = \vone_m$, it follows that monomal $\x^{\Dee \tr{\delta}{e}} \notin \bracket{\m}{e}$.

Combining this with \eqref{expansionf: e}, shows that $f^{p^e | \tr{\ddelta} | } \notin \bracket{\m}{e}$,  and Remark \ref{Flatness: R} then shows that $f^{p^d \tr{\ddelta}{e}} \notin \bracket{\m}{d}$ for all $d \geq e$.  Lemma \ref{TypicalFPT: L} shows  that $\fpt{f} \geq | \tr{\ddelta}{e} |$ for all $e$, and the claim follows by letting $e \to \infty$.
\end{proof}

\begin{DiagonalFPT2}
If $L < \infty$, then $\fpt{f} = |\tr{\ddelta}{L}|  + \frac{1}{p^L}$.
\end{DiagonalFPT2}
\begin{proof}
The estimate for $F$-pure thresholds given in \cite[\MPT]{Polynomials} implies that \begin{equation} \label{inequality} \fpt{f} \geq \tr{\ddelta}{L} + \frac{1}{p^L}.  \end{equation}  

\noindent If the inequality in \eqref{inequality} is strict, then Lemma \ref{TypicalFPT: L} implies there exists $e \geq L$ such that 
\begin{equation}
\label{diagfpt2: e} \(f^{p^L |\tr{\ddelta}{L}| + 1}\)^{p^{e-L}}  = 
f^{ p^e |\tr{\ddelta}{L}| + p^{e-L}} \notin \bracket{\m}{e}.
\end{equation}

By Remark \ref{Flatness: R}, it follows that $f^{p^L |\tr{\ddelta}{L}| + 1} \notin \bracket{\m}{L}$.  Applying \eqref{expansionf: e}  shows there exists $\k \in \mathbb{N}^n$ such that $| \k | = p^L |\tr{\ddelta}{L}| + 1$ and $\x^{\Dee \k} \notin \bracket{\m}{L}$.  This last condition implies that $\frac{1}{p^L} \cdot \k \vl \ddelta$, and applying Lemma \ref{Truncation: L} then shows $\frac{1}{p^L} \cdot \k \vleq \tr{\ddelta}{L}$.  Thus, $| \tr{\ddelta}{L}| + \frac{1}{p^L} = \frac{1}{p^L} \cdot \k \leq | \tr{\ddelta}{L}|$, a contradiction.  We conclude that equality holds in \eqref{inequality}, and so we are done.
\end{proof}

\section{Test ideals of diagonal hypersurfaces}
\label{DiagonalTI: S}

We now prove Theorem \ref{DiagonalTestIdeal: T} in three parts.   As before, we assume $f$ is a $\mathbb{L}^{\ast}$-linear combination of $x_1^{d_1}, \cdots,  x_n^{d_n}$: $f = \sum_{i=1}^n u_i x_i^{d_i}$.  We also continue to adopt Notation \ref{DeltaNotation: N}.

\begin{DTI1}
If $\fpt{f} = 1$, then $\test{R}{f}{\fpt{f}} = (f)$.
\end{DTI1}

\begin{proof}
Note that $f^{p^e} = f^{p^e} \cdot 1$, and that $1 \in \Basis{e}$.  This, $\Coeff{1}{f^p}{e} = f$ while $\Coeff{\mu}{f^p}{e} = 0$ for all $1 \neq \mu \in \Basis{e}$.  It follows from this, and Lemma \ref{Stabilization: L}, that $(f) = \pideal{f}{p^e}{e} = \test{R}{f}{1}$.
\end{proof}

To prove the remaining parts of Theorem \ref{DiagonalTestIdeal: T}, we will need Corollary \ref{ProjectionIdeal: C}  below. 

\begin{Lemma}
\label{BasisElementLemma}  The natural number $d_i \( p^e \tr{\delta_i}{e} + 1-p^e \delta_i \)$ is less than $d_i$.  In particular, if $d_i \leq p^e$, then $\Dee \( p^e \tr{\ddelta}{e} + ( 1 - p^e \delta_i) \cdot \vee_i\) \vl p^e \cdot \vone_n$. 
\end{Lemma}

\begin{proof}  As $\tr{\frac{1}{d}}{e} < \frac{1}{d}$, it follows that $d_i \( p^e \tr{\delta_i}{e} + 1-p^e \delta_i \) < d_i \( p^e \delta_i + 1 - p^e \delta_i\) = d_i$.
\end{proof}
 
 In the proof of Lemma \ref{Coefficient: L}, we use $\dotp$ to denote the standard dot product on $\mathbb{R}^n$.

\begin{Lemma} 
\label{Coefficient: L}
Suppose that $d_i < p^e$ and that $d_i$ is not a power of $p$.  By Lemma \ref{BasisElementLemma}, $\mu_i : = \x^{\Dee \( p^e \tr{\ddelta}{e} + ( 1 - p^e \delta_i) \cdot \vee_i\)} \in \Basis{e}$ and $\Coeff{\mu_i}{f^{p^e | \tr{\ddelta}{e} | + 1}}{e} = \( \binom{ | p^{\e} \tr{\ddelta}{\e}| + 1}{p^{\e} \tr{\ddelta}{\e} + \vee_i} \uu^{p^{\e} \tr{\ddelta}{\e} + \vee_i} \)^{1/p^e} \cdot x_i$.
\end{Lemma}

\begin{proof}  
To calculate $\Coeff{\mu_i}{f^{p^e | \tr{\ddelta}{e} | + 1}}{e}$, we must determine which (possibly) supporting monomials of $f^{p^e | \tr{\ddelta}{e}| + 1}$ are $R^{p^e}$ multiples of $\mu_i$.  A monomial satisfying this condition is of form $\x^{\Dee \k}$ for some $\k \in \mathbb{N}$ with $| \k | = p^e | \tr{\ddelta}{e} | + 1$ such that $\Dee \k = p^e \bold{a} + \Dee \( p^e \tr{\ddelta}{e} + ( 1 - p^e \delta_i) \cdot \vee_i\)$ for some vector $\bold{a}$.  Applying $\DDelta = \Dee^{-1}$ then shows that 
\begin{equation}
\label{coeffl1: e} 
\k = p^e \DDelta \bold{a} + p^e \tr{\ddelta}{e} + (1-p^e \delta_i) \cdot \vee_i.
\end{equation}

If $a_i = 0$, \eqref{coeffl1: e} shows that $k_i = p^e \tr{\delta_i}{e} + 1 - p^e \delta_i$, so that $p^e \delta_i \in \mathbb{N}$, which contradicts the assumption that $d_i$ is not a power of $p$.  Thus, $a_i \geq 1$. By summing the equation appearing in \eqref{coeffl1: e}, we see that $p^e | \tr{\ddelta}{e} | + 1 = | \k | = p^e \ddelta \dotp \bold{a} + p^e | \tr{\ddelta}{e} | + 1 - p^e \delta_i$, and so 
\begin{equation} \ddelta \dotp \( \bold{a} - \vee_i \) = 0.\end{equation}  As $a_i \geq 1$, $\bold{a} - \vee_i \vgeq \0$, and as the entries of $\ddelta$ are non-zero, follows from \eqref{coeffl1: e} that $\bold{a} = \vee_i$. Substituting this into \eqref{coeffl1: e} shows that the only (possibly) supporting monomial of $f^{p^e | \tr{\ddelta}{e} | + 1}$ that is an $R^{p^e}$-multiple of $\mu_i$ is 
$ \x^{\Dee \( p^e \tr{\ddelta}{e} + \vee_i \)} = \x^{\Dee \( p^e \tr{\ddelta}{e} + (1-p^e \delta_i)\vee_i \)}  \cdot \x^{\Dee p^e \delta_i} = \mu_i \cdot x_i^{p^e}$.
\end{proof}

\begin{Lemma}
\label{NZCoefficient: C}
If $ \sum \limits_{i=1}^n \digit{\delta_i}{e} \leq p-2$ and $\binom{ | p^{\e} \tr{\ddelta}{\e}|}{p^{\e} \tr{\ddelta}{\e}} \neq 0 \bmod p$, then $\binom{ | p^{\e} \tr{\ddelta}{\e}| + 1}{p^{\e} \tr{\ddelta}{\e} + \vee_i} \neq 0 \bmod p$.
\end{Lemma}
\begin{proof}  Note that $\digit{\delta_i}{e} \leq  \sum_{i=1}^n \digit{\delta_i}{e} \leq p-2$, which implies that both
$p^{\e} \tr{\delta_i}{e} +1 \equiv \digit{\delta_i}{e} + 1 $ and $p^{\e} |\tr{\ddelta}{\e} | + 1 \equiv \sum_{i=1}^n \digit{\delta_i}{e}$ are non-zero mod $p$.  The claim by reducing the equality $ \( p^{\e} \tr{\delta_i}{\e} + 1 \) \cdot \binom{ | p^{\e} \tr{\ddelta}{\e}| + 1}{p^{\e} \tr{\ddelta}{\e} + \ve_i} = \binom{ | p^{\e} \tr{\ddelta}{\e} |}{p^{\e} \tr{\ddelta}{\e}} \cdot \( | p^{\e} \tr{\ddelta}{\e} | + 1 \)$ mod $p$. 
\end{proof}

\begin{Corollary}  
\label{ProjectionIdeal: C}  Suppose that $d_i < p^e$ and is not a power of $p$.  If $\sum_{i=1}^n \digit{\delta_i}{e} \leq p-2$ and $\binom{ | p^{\e} \tr{\ddelta}{\e}|}{p^{\e} \tr{\ddelta}{\e}} \neq 0 \bmod p$, then $x_i \in \pideal{f}{p^{\e} | \tr{\ddelta}{e} | + 1}{\e}$.
\end{Corollary}

\begin{proof}
This follows immediately from Lemmas \ref{Coefficient: L} and \ref{NZCoefficient: C}.
\end{proof}

\begin{DTI2}
If $\fpt{f} = | \ddelta | < 1$, then $\test{R}{f}{\fpt{f}} = \m$.
\end{DTI2}

\begin{proof}
By Theorem \ref{DiagonalFPT: T}, the entries of $\ddelta$ add without carrying (in base $p$), so that no $d_i$ is a $p^{\th}$ power (for else carrying would be necessary) and $\binom{| p^e \tr{\ddelta}{e} | }{p^e \tr{\ddelta}{e}} \neq 0 \bmod p \text{ for all } e \geq 1$, by Lemma \ref{Lucas: L}.  As no $d_i$ is a  $p^{\th}$ power and $| \ddelta | < 1$, the denominator of $| \ddelta |$ is also not a $p^{\th}$ power, and applying Remark \ref{AddingDigits: R} shows $\sum_{i=1}^n \digit{\delta_i}{e} = \digit{|\ddelta|}{e} < p-1$ for infinitely many $e$.    Choose such an $e$ so that additionally every $d_i$ is less than $p^e$ and \[ \test{R}{f}{|\ddelta|} = \pideal{f}{\up{p^e | \ddelta | }}{e} = \pideal{f}{p^e | \tr{\ddelta}{e} | + 1}{e},  \] where we have used Lemma \ref{Truncation: L} to obtain the equality $\up{p^e | \ddelta |} = p^e \tr{|\ddelta|}{e} + 1 = p^e | \tr{\ddelta}{e} | + 1$.  Applying Corollary \ref{ProjectionIdeal: C} then shows $(x_1, \cdots, x_n) \subseteq \pideal{f}{p^e |\tr{\ddelta}{e}| +1}{e} = \test{R}{f}{| \ddelta |}$. 
\end{proof}

\begin{DTI3}
If $\fpt{f} < \min \set{ 1, |\ddelta| }$ and $p > \max \set{d_1, \cdots, d_n}$, then $\test{R}{f}{\fpt{f}} = \m$ .
\end{DTI3}

\begin{proof}
Let $L = \max \{ N : \sum_{i=1}^n \digit{\delta_i}{e} \leq p-1 \text{ for } 0 \leq e \leq N \}$, and set $\lambda : = | \tr{\ddelta}{L} | + \frac{1}{p^L}$. By definition, $\digit{\lambda}{e} = \sum_{i=1}^n \digit{\delta_i}{e}$ for $0 \leq e \leq L$ while $\digit{\lambda}{e} = p-1$ for $e \geq L+1$.  As  $\lambda < 1$, there exists $1 \leq l \leq L$ such that $\digit{\lambda}{l} = \sum_{i=1}^n \digit{\delta_i}{l} \leq p-2$ and $\digit{\lambda}{e} = p-1$ for $e \geq l$.  By Lemma \ref{Lucas: L}, our choice of $\l$ guarantees that $\binom{ p^l | \tr{ \ddelta}{l} |}{p^l \tr{\ddelta}{l} } \neq 0 \bmod p$.  As each $d_i$ is less than $p$, Corollary \ref{ProjectionIdeal: C} and Lemma \ref{Stabilization: L} combine to show that $(x_1, \cdots, x_n) \subseteq \pideal{f}{p^l | \tr{\ddelta}{l} | + 1}{l} = \test{R}{f}{\lambda}$.
\end{proof}

\section{On (higher) $F$-jumping numbers of Fermat hypersurfaces}
\label{JumpingNumbers: S}

\begin{Notation}
We now assume $f$ is a $\mathbb{L}^{\ast}$-linear combination of $x_1^{d_1}, \cdots,  x_n^{d_n}$, and write $f = u_1 x_1^{d} + \cdots + u_d x_d^{d}$. We continue to use $\delta$ to denote $\frac{1}{d}$.
\end{Notation}

\begin{Remark}
\label{deltaidentity: R}
 Supposes $p > d$, and fix integers $\emw \geq 1$ and $1 \leq a < d$ such that $p = d \cdot \emw + a$.  Isolating $\delta = \frac{1}{d}$ in this equation shows that $\delta = \frac{\emw}{p} + \frac{a}{d} \cdot \frac{1}{p} = \frac{\emw}{p} + (a \delta) \cdot \frac{1}{p}$.  From this, we conclude that $\digit{\delta}{1} = \emw \text{ and that } \digit{\delta}{e+1} = \digit{\( a \delta \)}{e} \text{ for all }e \geq 1$.  
\end{Remark}

The following Lemma will be key in proving Corollary \ref{FermatFPT: C}.

\begin{Lemma}
\label{DigitsTrick: L}
Suppose that $p > d > 2$ and $p \equiv  a \bmod d$.  If $a \geq 2$, then $(d-1) \cdot \digit{\delta}{2} \geq p+1$.
\end{Lemma}

\begin{proof}
Suppose, by means of contradiction, that $(d-1) \cdot \digit{\delta}{2} \leq p$.  As $p$ is prime and both $d-1$ and $\digit{\delta}{2}$ are less than $p$, equality cannot hold.  In particular,  
\begin{equation} \label{digitstrick1: e} (d-1) \cdot \digit{\delta}{2} \leq p-1. \end{equation}

By Remark \ref{deltaidentity: R}, we know $\frac{a}{d} = \sum_{e \geq 1} \frac{\digit{\delta}{e+1}}{p^e}$, and combining this observation with \eqref{digitstrick1: e} shows 
\begin{align}
\label{digitstrick2: e}
(d-1) \cdot \frac{a}{d} = \frac{(d-1) \cdot \digit{\delta}{2}}{p} + (d-1) \cdot \sum_{e=2}^{\infty} \frac{ \digit{\delta}{e+1}}{p^e}  & \leq \frac{(d-1) \cdot \digit{\delta}{2} }{p} + \frac{d-1}{p} \\ & \leq \frac{p-1}{p} + \frac{d-1}{p}  = 1 + \frac{d-2}{p}.  \notag
\end{align}

However, as $a \geq 2$, $(d-1) \cdot \frac{a}{d} \geq (d-1) \cdot \frac{2}{d} =  1 + \frac{d-2}{d}$, and 
comparing this with \eqref{digitstrick2: e}  shows $\frac{d-2}{d} \leq \frac{d-2}{p}$, which implies that $p \leq d$, a contradiction.
\end{proof}

\begin{FermatFPTCorollary}  We have the following formula for $\fpt{f}$:
\[ \fpt{f} = \begin{cases}  \frac{1}{p^{\ell}} & \text{if } p^{\ell} \leq d < p^{\ell + 1} \text{ for some } \ell \geq 1 \\ 1 - \frac{a-1}{p} & \text{if } 0 < d < p \text{ and } p \equiv a \bmod d \text{ with } 1 \leq a < d \end{cases} \]
\end{FermatFPTCorollary}

\begin{proof}

If $p^{\ell} \leq d < p^{\ell + 1}$, then $\frac{1}{p^{\ell+1}} < \delta \leq \frac{1}{p^{\ell}}$.  Consequently, $\digit{\delta}{e} \text{ for } 1 \leq e \leq \ell \text{ and } \digit{\delta}{l+1} \neq 0$.  Adding $d$ copies of $\digit{\delta}{\ell+1}$ yields $d \cdot \digit{\delta}{\ell+1} \geq d \geq p^{\ell} \geq p$.  In the notation of Theorem \ref{DiagonalFPT: T}, we have that $L = \ell$, and as $\tr{\delta}{\ell} = 0$, $\fpt{f} = d \cdot \tr{\delta}{\ell} + \frac{1}{p^{\ell}} = \frac{1}{p^{\ell}}$.

We now assume that $p>d$.  If $a=1$, the identities in Remark \ref{deltaidentity: R} imply $ \digit{\delta}{e} = \emw$ for all $e \geq 1$. As $d \cdot \digit{\delta}{e} = d \cdot \emw = p-1$, it follows that $d$ copies of $\delta$ add without carrying.   By Theorem \ref{DiagonalFPT: T}, $\fpt{f}  = 1$.  Suppose now that $a \geq 2$ (which automatically implies $d > 2$).  Note that $d \cdot \digit{\delta}{1} = d \cdot \emw = p-a$ while $d \cdot \digit{\delta}{2} > p$ by Lemma \ref{DigitsTrick: L}.  By Theorem \ref{DiagonalFPT: T}, $\fpt{f} = d \cdot \tr{\delta}{1}  + \frac{1}{p} = \frac{d \cdot \emw}{p} + \frac{1}{p} = \frac{p-a+1}{p}$.
\end{proof}

In order to prove Theorem \ref{JumpingNumbers: T}, we will need the following lemmas.

\begin{Lemma}
\label{NoGathering: L}
If $d$ is not a $p^{\th}$ power, then $x_i \in \pideal{f}{N}{e}$ if and only if $0 \vleq d \cdot \k - p^e \cdot \vee_i \vl p^e \cdot \vone_d$ and $\binom{N}{\k} \neq 0 \bmod p$ for some $\k$ with $|\k| =N$.
\end{Lemma}

\begin{proof}
For every $\k \in \mathbb{N}^d$, there is a unique element $\cc_{\k} \in \mathbb{N}^d$ such that $0 \vleq d \cdot \k - p^e \cdot \cc_{\k} \vl p^e \cdot \vone_d$.  If we set $\mu_{\k} : = \x^{d \cdot \k - p^e \cdot \cc_{\k} }$, it follows that $\mu_k \in \Basis{e}$ and that $\x^{d \cdot \k} = \x^{p^e \cdot \cc_{\k}} \mu_{\k}$.  Thus,  
\begin{equation}
 \label{nogathering1: e}
 f^N = \sum_{|\k| = N} \binom{N}{k} \uu^{\k} \x^{d \cdot \k} = \sum_{|\k| = N} \( \(\binom{N}{\k} \uu^{\k}\)^{1/p^e} \x^{\cc_{\k}}\)^{p^e} \mu_{\k}.  
\end{equation}

Let $I$ denote the ideal generated by the elements  $\(\binom{N}{\k} \uu^{\k}\)^{1/p^e} \x^{\cc_{\k}}$.  Apparently,  \eqref{nogathering1: e} shows that $f^N$ is in $\bracket{I}{e}$, and applying Lemma \ref{LittleL: L} then shows $\pideal{f}{N}{e} \subseteq I$.  If $x_i \in \pideal{f}{N}{e}$, then $x_i \in I$, and so $x_i$ must be a unit multiple of one of the monomial generators of $I$.  We conclude that $x_i = \x^{\cc_{\k}}$ for some $\k$ with $\binom{N}{\k} \neq 0 \bmod p$.

Next, suppose that $0 \vleq d \cdot \k - p^e \cdot \vee_i \vl p^e \cdot \vone_d$ and $\binom{N}{\k} \neq 0 \bmod p$ for some $\k$ with $|\k| =N$, so that $\x^{d \cdot \k} = x_i^{p^e} \mu_{\k}$ is a supporting monomial of $f^N$.  To show that $x_i \in \pideal{f}{N}{e}$, it suffices to show that $\x^{d \cdot \k}$ is the only supporting monomial of $f^N$ that is an $R^{p^e}$-multiple of $\mu_{\k}$.  Let $\x^{d \cdot \bold{\kappa}}$ be another such monomial, so that $\x^{d \cdot \bold{\kappa}} = \x^{p^e \cc}\mu_{\k} \text{ and } \x^{d \cdot \k } = x_i^{p^e} \mu_{\k}$.  Solving for $\mu_{\k}$ in these expressions shows $\mu_{\k} = \x^{d \cdot \k - p^e \cc} = \x^{d \cdot \bold{\kappa} - p^e \vee_i}$, and so
\begin{equation}
\label{nogathering2: e}
d \cdot (\k - \bold{\kappa} ) = p^e \cdot (\cc - \vee_i).
\end{equation}

As $| \k | = | \bold{\kappa} | = N$, it follows from \eqref{nogathering2: e} that $| \cc | = | \vee_i | = 1$, so that $\cc = \vee_j$ for some $j$.  If $j \neq i$, then \eqref{nogathering2: e} shows that $d(k_j - \kappa_j) = p^e$, which contradicts the assumption that $d$ is not a $p^{\th}$ power.  Thus, $\cc = \vee_i$, and so $\k = \bold{\kappa}$ by \eqref{nogathering2: e}.
\end{proof}

\begin{Notation}
From now on, we assume $p = d \cdot \emw + a$ for some $\emw \geq 0$ and $1 \leq a < d$.
\end{Notation}

\begin{Lemma}
\label{Bounds: L}  \ We have the following inequalities:   
\begin{enumerate}
\item \label{Bounds: L2} $p < d( 2 \emw + \up{2a \delta} -1) < 2p$.
\item \label{Bounds: L1} If $p < a (d -1)$, then $(d+1) \cdot \emw + \up{2a \delta} \leq p$.
\item \label{Bounds: L3} If $p> a(d-1)$, then $p < d ( \emw + a -1 ) < 2p$.
\end{enumerate}
\end{Lemma}

\begin{proof}
The first point follows by applying the inequality $2a\delta \leq \up{2a \delta} < 2a\delta + 1$ and the identity $p = d \cdot \emw + a$.  For the second point, note that $p = d \cdot \emw + a < ad -a$
by hypothesis, and it follows that $\emw + 2a\delta < a$.  Adding $d \cdot \emw$ to both sides yields $(d+1) \cdot \emw + 2a\delta < a + d \cdot \emw = p$.  The proof of the third point is similar, and is left to the reader.
\end{proof}


\begin{JN1}
If $a \geq 2$ and $p < a(d-1)$, then 
\begin{enumerate}
\item $\test{R}{f}{\( \frac{(d+1) \cdot \emw + \up{2a \delta}-1}{p} + \frac{p-1}{p^2} + \cdots + \frac{p-1}{p^e} \) } = \m$ for all $e \geq 1$, and 
\item $\test{R}{f}{\( \frac{(d+1) \cdot \emw + \up{2a \delta}}{p} \) } \neq \m$.
\end{enumerate}
In particular, $\frac{(d+1) \cdot \emw + \up{2a \delta}}{p} \in (0,1]$ is an $F$-jumping number of $f$. 
\end{JN1}

\begin{proof} By Lemma \ref{Bounds: L},  $\frac{(d+1) \cdot \emw + \up{2a \delta}}{p} \in (0,1]$ .  By Lemma \ref{DigitsTrick: L}, $(d-1) \cdot \digit{\delta}{2} \geq p+1$.  Thus, there exists non-negative integers $l_1, \cdots, l_{d-1}$ such that $\sum l_i = p-1$ and $l_i \leq \digit{\delta}{2}$ for $1 \leq i \leq d-1$, with the  inequality being strict for at least one  $i$, which we are free to choose.  In what follows, we assume that $l_{d-1} < \digit{\delta}{2}$.  Fix $e \geq 3$, and set 
\[ \llam_e:= \( \frac{\emw}{p} + \frac{l_1}{p^2}, \cdots, \frac{\emw}{p} + \frac{l_{d-2}}{p^2}, \frac{\emw}{p} + \frac{l_{d-1}}{p^2} + \frac{p-1}{p^3} + \cdots + \frac{p-1}{p^e}, \frac{2 \emw +\up{2a \delta} -1}{p} \). \]

By Remark \ref{deltaidentity: R}, $\digit{\delta}{1} = \emw$, and as $l_{d-1} < \digit{\delta}{2}$, the first $d-1$ entries of $\llam_e$ are less than or equal to $\tr{\delta}{2}$.  Set $\k:= p^e \llam_e$.  It follows that the first $d-1$ entries of $d \cdot \k$ are less than or equal to $d \cdot p^e \tr{\delta}{2}$, and thus strictly less than $p^e$ while, by Lemma \ref{Bounds: L}, the last entry of $d \cdot \k$ is \emph{strictly} between $p^e$ and $2p^e$.  Thus, $\0 \vleq d \cdot \k - p^e \cdot \vee_d \vl p^e \cdot \vone_d$.  
By construction, the entries of $\k$ add without carrying (in base $p$), so $\binom{|\k|}{\k} \neq 0 \bmod p$.  Finally, 
\begin{align*} |\k| &= p^e \cdot \( \frac{(d+1) \cdot \emw + \up{2a \delta} - 1}{p} + \sum_{i=1}^{d-1} \frac{l_i}{p^2} + \frac{p-1}{p^3} + \cdots + \frac{ p-1}{p^e} \) \\ & = p^e \cdot \( \frac{(d+1) \cdot \emw + \up{2a \delta} - 1}{p} + \frac{p-1}{p^2} + \frac{p-1}{p^3} \cdots + \frac{ p-1}{p^e} \).
\end{align*}
We then apply Lemmas \ref{NoGathering: L} and \ref{Stabilization: L} to deduce that \[ x_d \in \pideal{f}{|\k|}{e} = \test{R}{f}{\( \frac{(d+1) \cdot \emw + \up{2a \delta} - 1}{p} + \frac{p-1}{p^2} + \cdots + \frac{p-1}{p^e} \) }.\]  As this argument is symmetric in the variables, the first claim follows.

We now show that $\test{R}{f}{\frac{(d+1) \cdot \emw + \up{2a \delta}}{p}} \neq \m$.  By way of contradiction, suppose that $x_1$ is in $\test{R}{f}{\frac{(d+1) \cdot \emw + \up{2a \delta}}{p}} = \pideal{f}{(d+1) \cdot \emw + \up{2a \delta}}{}$.  By Lemma \ref{NoGathering: L}, there exists $\k \in \mathbb{N}^d$ with $| \k | = (d+1) \cdot \emw + \up{2a \delta}$ such that $0 \vleq d \cdot \k - p \cdot \vee_1 \vl p \cdot \vone_d$.  Restated, $0 \vleq \frac{1}{p} \cdot \k - \delta \cdot \vee_1 \vl \delta \cdot \vone_d$, and applying Lemma \ref{Truncation: L}  shows $\frac{1}{p} \cdot k_i \leq \tr{\delta}{1} = \frac{\digit{\delta}{1}}{p} = \frac{\emw}{p}$ for all $2 \leq i \leq d$.  These same inequalities also show $\frac{1}{p} \cdot k_1 < 2 \delta$, and summing these bounds shows
\[   (d+1) \cdot \emw +  \up{2a\delta} = | \k | < (d-1) \cdot \emw + 2 p \delta. \]  Gathering the multiples of $\emw$ and multiplying through by $d$ implies $2 d \cdot \emw + d \cdot \up{2a\delta} < 2p$, which this is impossible as $2d \cdot \emw  +d \cdot \up{2a \delta} \geq 2d\cdot \emw + d \cdot 2a\delta = 2( d \cdot \emw + a) = 2p$.  We conclude that $x_1$ (and by symmetry, no variable) is in $\test{R}{f}{\frac{(d+1) \cdot \emw + \up{2a \delta}}{p}}$. 
\end{proof}

\begin{JN2}
\label{JumpingNumber2: L}
If  $a \geq 2$ and  $p> a (d-1)$, then $\test{R}{f}{ \( \frac{p-1}{p} + \cdots + \frac{p-1}{p^e} \) } = \m$ for all $e \geq 1$.  In particular,  the only $F$-jumping numbers of $f$ in $(0,1]$ are $\fpt{f}$ and $1$.
\end{JN2}

\begin{proof}
As in the proof of Theorem \ref{JumpingNumbers: T}: Part I, Lemma \ref{DigitsTrick: L} guarantees there exists non-negative integers $l_1, \cdots, l_{d-1}$ such that $\sum_{i=1}^{d-1} l_i = p-1$ and $l_i \leq \digit{\delta}{2}$ for $1 \leq i \leq d-1$, with at least one inequality being strict.  We again assume $l_{d-1} < \digit{\delta}{2}$.  Fix $e \geq 3$, and let 
\[ \llam_e:= \( \frac{\emw}{p} + \frac{l_1}{p^2}, \cdots, \frac{\emw}{p} + \frac{l_{d-2}}{p^2}, \frac{\emw}{p} + \frac{l_{d-1}}{p^2} + \frac{p-1}{p^3} + \cdots + \frac{p-1}{p^e}, \frac{\emw+a-1}{p} \). \]
Set $\k = p^e \cdot \llam_e$. By construction, the entries of $\k$ add without carrying, so $\binom{|\k|}{\k} \neq 0 \bmod p$.  As in the proof of Part I of Theorem \ref{JumpingNumbers: T}, one may verify that $\0 \vleq d \cdot \k - p^e \cdot \vee_d \vl p^e \cdot \vone_d$, and
\begin{align*} |\k| &= p^e \cdot \( \frac{d \cdot \emw + a - 1}{p} + \sum_{i=1}^{d-1} \frac{l_i}{p^2} + \frac{p-1}{p^3} + \cdots + \frac{ p-1}{p^e} \) \\ & = p^e \cdot \( \frac{p-1}{p} + \frac{p-1}{p^2} + \frac{p-1}{p^3} \cdots + \frac{ p-1}{p^e} \). \end{align*} Once more,  Lemmas \ref{NoGathering: L} and \ref{Stabilization: L} imply $x_d \in \pideal{f}{|\k|}{e} = \test{R}{f}{\( \frac{p - 1}{p} + \frac{p-1}{p^2} + \cdots + \frac{p-1}{p^e} \) }$. By the symmetry of this argument, we conclude that $\m \subseteq \test{R}{f}{ \( \frac{p-1}{p} + \cdots + \frac{p-1}{p^e} \) }$.
\end{proof}


\bibliographystyle{alpha}
\bibliography{refs}

\end{document}